\documentclass{compositio}
\usepackage{amssymb,amsmath,amsthm,mathrsfs,verbatim, amscd, amsthm}
\usepackage{enumitem}

\newtheorem{theorem}{Theorem}
\newtheorem{lemma}[theorem]{Lemma}

\newtheorem{proposition}[theorem]{Proposition}

\newtheorem*{remark}{Remark}
\theoremstyle{remark}

\newcommand{\F}{\mathbb{F}_p}
\newcommand{\Ff}{\mathbb{F}^\times_p}
\newcommand{\ds}{\displaystyle}
\newcommand{\g}[1]{\left(\frac{#1}{p}\right)}
\newcommand{\sm}[4]{\left(\begin{smallmatrix}#1&#2\\ #3&#4 \end{smallmatrix}
\right)}

\newcommand{\Q}{\mathbb{Q}}
\newcommand{\Gal}{\textrm{Gal}(\overline{\Q}/\Q)}

\newcommand{\SL}{{\text {\rm SL}}}

\def\Z{\mathbb{Z}}
\def\E{\tilde{E}_t(\F)}

\renewcommand{\contentsname}{Contents\\{\footnotesize\normalfont(A table
of contents should normally not be included)}}
\begin{document}

\title{Diophantine $m$-tuples in finite fields and modular forms}
\author{Andrej Dujella}
\email{duje@math.hr}
\address{Department of Mathematics, Faculty of Science, University of Zagreb, Bijeni\v{c}ka cesta 30, 10000 Zagreb, Croatia}
\author{Matija Kazalicki}
\email{matija.kazalicki@math.hr}%
\address{Department of Mathematics, Faculty of Science, University of Zagreb, Bijeni\v{c}ka cesta 30, 10000 Zagreb, Croatia}

\classification{11G25, 11F80 (primary), 	11T24 (secondary)}
\keywords{Diophantine $m$-tuples, modular forms, finite fields, elliptic curves}

\begin{abstract}
For a prime $p$, a Diophantine $m$-tuple in $\F$ is a set of $m$ nonzero elements of $\F$ with the property that the product of any two of its distinct elements is one less than a square.

In this paper, we present formulas for the number $N^{(m)}(p)$ of Diophantine $m$-tuples in $\F$ for $m=2,3$ and $4$. Fourier coefficients of certain modular forms appear in the formula for the number of Diophantine quadruples.

We prove that asymptotically $N^{(m)}(p)=\ds \frac{1}{2^{m \choose 2 }}\frac{p^m}{m!} + o(p^m)$, and also show that if $p>2^{2m-2}m^2$, then there is at least one Diophantine $m$-tuple in $\F$.
\end{abstract}

\maketitle

\section{Introduction}
A Diophantine $m$-tuple is a set of $m$ positive
integers with the property that the product of any two of its distinct
elements is one less than a square. If a set of nonzero rationals
has the same property, then it is called
a rational Diophantine $m$-tuple.
Diophantus of Alexandria found the first example of a rational Diophantine quadruple
$\{1/16, 33/16, 17/4, 105/16\}$, while the first Diophantine quadruple in integers was
found by Fermat, and it was the set $\{1,3,8,120\}$.
It was proved in \cite{D-crelle} that an integer Diophantine sextuple does not exist and that there are
only finitely many such quintuples. A folklore conjecture is that there does not exist an integer Diophantine quintuple. On the other hand, it was shown in \cite{DKMS} that there are infinitely many rational Diophantine sextuples (for another construction see \cite{DK}), and it is not known if there are rational Diophantine septuples. For a short survey on Diophantine $m$-tuples see \cite{Notices}.

One can study Diophantine $m$-tuples over any commutative ring with unity. In this paper, we consider Diophantine $m$-tuples in finite fields $\F$, where $p$ is an odd prime. In this setting, it is natural to ask about the number $N^{(m)}(p)$ of Diophantine $m$-tuples with elements in $\F$ (we consider $0$ to be a square in $\F$).

Since half of the elements of $\Ff$ are squares, heuristically, one expects that a randomly chosen $m$-tuple of different elements in $\Ff$ will have the Diophantine property with probability $\ds \frac{1}{2^{m \choose 2 }}$, i.e. we expect $N^{(m)}(p)=\ds \frac{1}{2^{m \choose 2 }}\frac{p^m}{m!} + o(p^m)$. We prove this asymptotic formula at the end of Section \ref{sec:gen}.

The main theorem of the paper gives an exact formula for the number of Diophantine quadruples $N^{(4)}(p)$ given in terms of the Fourier coefficients of the following modular forms.

Let
\begin{eqnarray*}
\ds f_1(\tau)&=&\sum_{n=1}^\infty a(n) q^n \in S_2\left(\Gamma_0(32)\right),\\
\ds f_2(\tau)&=&\sum_{n=1}^\infty b(n) q^n \in S_3\left(\Gamma_0(8), \left(\frac{-2}{\bullet} \right)\right),\\
\ds f_3(\tau)&=&\sum_{n=1}^\infty c(n) q^n \in S_3\left(\Gamma_0(16), \left(\frac{-4}{\bullet} \right)\right),\\
\ds f_4(\tau)&=&\sum_{n=1}^\infty d(n) q^n \in S_4\left(\Gamma_0(8)\right),\\
\ds f_5(\tau)&=&\sum_{n=1}^\infty e(n) q^n \in S_5\left(\Gamma_0(4),\left(\frac{-4}{\bullet} \right) \right),\\
\end{eqnarray*}
 be (the unique rational) newforms in the corresponding spaces of modular forms.
Here $S_k(\Gamma_0(N), \chi)$ denotes the space of cusp forms of weight $k$, level $N$ and nebentypus $\chi$.
Note that all modular forms except $f_4(\tau)$ are CM forms so we have explicit formulas for their Fourier coefficients which are given in Section \ref{sec:forms}.

\begin{theorem} \label{thm:main} Let $p$ be an odd prime. Denote by $q(p)=e(p)-6d(p)+24b(p)-24c(p)$. Then, the number $N^{(4)}(p)$ of Diophantine quadruples over $\F$ is given by the following formula:
\[
	N^{(4)}(p)=
	\begin{cases}
		\frac{1}{24\cdot 64}\left(p^4-24p^3+206p^2-650p + 477+q(p)\right), \textrm{ if } p \equiv 1 \pmod{8},\\
		\frac{1}{24\cdot 64}\left(p^4 - 24p^3 + 236p^2 - 1098p + 1761 + q(p) \right), \textrm{ if } p \equiv 3 \pmod{8},\\
		\frac{1}{24\cdot 64} \left( p^4 - 24p^3 + 206p^2 - 698p + 573+q(p)\right), \textrm{ if } p \equiv 5 \pmod{8},\\
		\frac{1}{24\cdot 64} \left(p^4 - 24p^3 + 236p^2 - 1050p + 1761 +q(p)\right), \textrm{ if } p \equiv 7 \pmod{8}.\\
	\end{cases}
\]
\end{theorem}

An elementary upper bound for the Fourier coefficients of cusp forms (see Chapter 5 of \cite{Iwa}) implies that $q(p) = O(p^{5/2})$, so we have $N^{(4)}(p)=\frac{1}{24\cdot 64}p^4+O(p^3)$, which is consistent with the heuristics mentioned earlier.

 In addition to this, using a more elementary approach of character sums, in Propositions \ref{prop:pairs} and \ref{prop:triples} we derive formulas for the number of Diophantine pairs $N^{(2)}(p)$ and the number of Diophantine triples $N^{(3)}(p)$ in $\F$. Already for $m=4$ this method becomes too involved.

For a general $m$ it is natural to ask how large $p$ must be so that there is at least one Diophantine $m$-tuple in $\F$. In Theorem \ref{tm:mtuples}, we prove that this is the case if $p>2^{2m-2}m^2$.

The rest of the paper is organized as follows.
In Section \ref{sec:cor}, we construct a correspondence between the set of Diophantine quadruples $\{a,b,c,d\}$ and the set of admissible triples $(Q_1, Q_2, Q_3)$ of $\F$-points on the curve $\mathcal{D}_t: (x^2-1)(y^2-1)=t$, for some $t \in \F^\times$ such that each admissible triple corresponds to one or two Diophantine quadruples. If $t\ne 0,1$, the curve $\mathcal{D}_t$ is birationally equivalent to the elliptic curve $E_t: V^2 = U^3 -2(t-2)U^2 + t^2 U$, with the distinguished point $R=(t,2t)$ of order $4$. Hence we identify $\mathcal{D}_t(\F)$ with $\E:=E_t(\F)\setminus\{\mathcal{O}, R, 2R, 3R\}$.

If $t\ne 0,1$ we say that the triple $(Q_1, Q_2, Q_3)$ of points on $\E$ is admissible if and only if $x(Q_1+Q_2+Q_3+R)$ is a square, and if for no two $Q_i$ and $Q_j$ with $i \ne j$,  we have that $Q_i=\pm Q_j+kR$, where $k \in \{0,1,2,3\}$. For the definition of admissibility when $t=1$ see the end of Section \ref{sec:cor}.

In Section \ref{sec:counting}, we find a formula for $N^{(4)}(p)$ by counting admissible triples on $\E$ for each $t$ (see Propositions \ref{prop:main} and \ref{prop:W1}). The formula can be written as a linear combination of sums of the form $\sum_{t\in X(\F)} P(t)^k$, where $X$ is one of the modular curves (for definitions see Section \ref{sec:curves} ) $$X_1(4), X_1(8), X(2,4), X(2,8), X(4,8)$$
and $P(t)$ is the number of $\F$-rational points on the fiber above $t$ of the universal elliptic curve over the modular curve $X$, and $k\in \{0,1,2,3\}$.

In Section \ref{sec:universal}, using universal elliptic curves over the modular curves introduced above, we define certain compatible families of $\ell$-adic Galois representations such that the trace of Frobenius $Frob_p$ under these representations is essentially equal to the sums above. On the other hand, these representations are isomorphic to the $\ell$-adic realisations of the motives associated to the spaces of cusps forms of weight $k+2$ on the corresponding groups, which enables us to express the traces of Frobenius in terms of the coefficients of the Hecke eigenforms in those spaces.

In Section \ref{sec:results}, using the methods from the previous section we calculate in Propositions \ref{prop:prva}-\ref{prop:zadnja} the sums from the formula for $N^{(4)}(p)$, and prove Theorem \ref{thm:main}.

By using character sums (Weil's estimates), in Section \ref{sec:gen} we obtain formulas for $N^{(2)}(p)$ and $N^{(3)}(p)$, and prove Theorem \ref{tm:mtuples} together with an asymptotic formula for $N^{(m)}(p)$.

\section{Correspondence}\label{sec:cor}

Let $\{a,b,c,d\}$ be a Diophantine quadruple with elements in $\F$, and let
\begin{align*}
ab+1=t_{12}^2, \quad ac+1&=t_{13}^2, \quad ad+1=t_{14}^2,\\
bc+1=t_{23}^2, \quad bd+1&=t_{24}^2, \quad cd+1=t_{34}^2.
\end{align*}

It follows that $(t_{12},t_{34},t_{13},t_{24},t_{14},t_{23}, t=abcd)\in {\F}^7$ defines a point on an algebraic variety $\mathcal{C}$ over $\F$  defined by the following equations:
\begin{align*}
(t_{12}^2-1)(t_{34}^2-1)&=t\\
(t_{13}^2-1)(t_{24}^2-1)&=t\\
(t_{14}^2-1)(t_{23}^2-1)&=t.
\end{align*}

Conversely, the points $(\pm t_{12},\pm t_{34},\pm t_{13},\pm t_{24},\pm t_{14},\pm t_{23}, t)\in \mathbb{F}^7_p$ on $\mathcal{C}$ determine two Diophantine quadruples $\pm\{a,b,c,d\}$ (or one if $\{a,b,c,d\}=\{-a,-b,-c,-d\}$), provided that the elements $a,b,c$ and $d$ are $\F$-rational, distinct and non-zero. Here, we take $a=\sqrt{(t_{12}^2-1)(t_{13}^2-1)/(t_{23}^2-1)}$ to be any square root, while $b, c$ and $d$ are defined using identities $ab+1=t_{12}^2$, $ac+1=t_{13}^2$ and $ad+1=t_{14}^2$. It follows from this definition and the equations defining $\mathcal{C}$ that $bc+1=t_{23}^2$, $bd+1=t_{24}^2$ and $cd+1=t_{34}^2$.
Also, if only one element of quadruple is $\F$-rational, then all the elements are $\F$-rational.

The projection $(t_{12},t_{34},t_{13},t_{24},t_{14},t_{23}, t) \mapsto t$ defines a fibration of $\mathcal{C}$ over the projective line, and the generic fiber is a cube of $\mathcal{D}_t: (x^2-1)(y^2-1)=t$. Any point on $\mathcal{C}$ corresponds to the three points $Q_1=(t_{12},t_{34})$, $Q_2=(t_{13},t_{24})$ and $Q_3=(t_{14}, t_{23})$ on $\mathcal{D}_t$. The elements of a quadruple corresponding to these three points are distinct if and only if no two of these points can be transformed from one to another by changing signs and switching coordinates (e.g. for the points $(t_{12}, t_{34})$, $(-t_{34}, t_{12})$ and $(t_{14},t_{23})$, we have that $a=d$).

The curve $\mathcal{D}_t$ for $t\in \F$ is birationally equivalent to the curve
$$E_t: V^2 = U^3 -2(t-2)U^2 + t^2 U.$$
The map is given by $U = 2(x^2-1)y+2x^2-(2-t)$,
and $V = 2Ux$.
The family $E_t$ over the $t$-line together with $R=(t,2t)$, the point of order $4$, is the universal elliptic curve over the modular curve $X_1(4)$(we identify $\mathbb{P}^1$ with $X_1(4)$ such that cusps of $X_1(4)$ correspond to $t=0, 1$ and $\infty$).
It is easy to see that the affine $\F$-points on the curve $\mathcal{D}_t$ are in $1-1$ correspondence with the set $\E := E_t(\F)\setminus \{\mathcal{O}, R, 2R, 3R \}$.

If $t\ne 0, 1$ the curve $E_t$ is an elliptic curve, so in our analysis of Diophantine quadruples we will naturally distinguish two cases $t=1$ and $t \ne 0, 1$ (note that $t=0$ would imply that one of the elements in quadruple $\{a,b,c,d\}$ is zero).

If $t=1$ then there is a singular point $(-1,0)$ on the curve $E_1: V^2=U(U+1)^2$ which corresponds to the point $(0,0)$ on $\mathcal{D}_1$.

If $Q\in \E$ is the nonsingular point that corresponds to the point $(x,y)\in \mathcal{D}_t(\F)$, then a direct calculation shows that the points $-Q$ and $Q+R$ correspond to the points $(-x,y)$ and $(y,-x)$ respectively. Hence the following lemma follows.

\begin{lemma}\label{lem:dis}
Let $t \in \F^{\times}$. The triple $(Q_1,Q_2,Q_3)\in \E^3$ corresponds to the quadruple whose elements are not distinct if and only if there are two nonsingular points, $Q_i$ and $Q_j$ with $i \ne j$ such that $Q_i=\pm Q_j+kR$, where $k \in \{0,1,2,3\}$ or if at least two points in the triple are singular.
\end{lemma}

A short calculation shows that for the point $(U,V)\in \tilde{E_t}(\F)$ corresponding to $(x,y)\in \mathcal{D}_t$ we have
\begin{equation*}
x^2-1=\left(\frac{V}{2U}\right)^2-1=T\left(\frac{U - t}{2U}\right)^2=:f((U,V)).
\end{equation*}
Since
\[
\begin{array}{ll}
a^2 &=\displaystyle \frac{f(Q_1)f(Q_2)f(Q_3)}{t} \equiv x(Q_1)x(Q_2)x(Q_3)t\equiv x(Q_1)x(Q_2)x(Q_3)x(R) \pmod{\F^{\times 2}}
\end{array}
\]
 for the rationality of $a$ it is enough to prove that $x(Q_1)x(Q_2)x(Q_3)x(R)$ is a square in $\F$.

If $t=1$ and $Q=(U,V)\in \tilde{E_1}(\F)$ is a nonsingular point, then $x(Q)=\frac{V^2}{(U+1)^2}$ is always a square in $\F$, hence the triple $(Q_1,Q_2,Q_3)\in \tilde{E_1}(\F)^3$ of distinct points corresponds to the quadruple whose elements are $\F$-rational if and only if $-1$ is a square in $\F$ (since $-1$ is $x$-coordinate of the singular point) or if all the points $Q_i$ are nonsingular.

\begin{lemma}\label{lem:rat}
If $t\ne 0,1$ then the triple $(Q_1,Q_2,Q_3)\in \E^3$ corresponds to the quadruple whose elements are $\F$-rational, if and only if
$$ Q_1+Q_2+Q_3+R \in \{\mathcal{O}, 2R\} \textrm{ or } x(Q_1+Q_2+Q_3+R) \textrm{ is a square}.$$
\end{lemma}
\begin{proof}
Since $(0,0)\in E_t(\F)$ is a point of order two, there is an elliptic curve $E_t^{'}$ defined over $\F$ and 2-isogeny $\phi':E_t \rightarrow E_t^{'}$ such that $\ker{\phi^{'}}=\langle (0,0) \rangle$. Denote by $\phi:E_t^{'} \rightarrow E_t$ the dual isogeny of $\phi^{'}$. Then there is a descent homomorphism $\delta_\phi: E_t(\F)/\phi(E_t^{'}(\F) \rightarrow H^1(\F,E_t^{'}[\phi])\cong\F^\times/\F^{\times 2}$ (see Section 2 of \cite{MS}), which maps point $(U,V)\mapsto U$ if $U\ne 0$, and points $2R=(0,0),\mathcal{O} \mapsto 1$.

It follows that
$$x(Q_1)x(Q_2)x(Q_3)x(R) \equiv \delta_\phi(Q_1+Q_2+Q_3+R) \pmod{\F^{\times 2}},$$
hence the claim follows.
\end{proof}

We call a triple $(Q_1, Q_2, Q_3) \in \E^3$ \textit{admissible} if it corresponds to a Diophantine quadruple. It follows from Lemma \ref{lem:dis} and Lemma \ref{lem:rat} that this holds if and only if the following is true

\begin{itemize}
	\item[a)] $t \ne 0,$
	\item[b)] there are no two nonsingular points $Q_i$ and $Q_j$ with $i\ne j$ such that $Q_i=\pm Q_j+kR$ for some $k\in \{0, 1, 2, 3 \}$,
	\item[c)] if $t\ne 0,1$ then $x(Q_1+Q_2+Q_3+R)$ is a square in $\F$ or if $t=1$ then all $Q_i$'s are nonsingular or $-1$ is a square in $\F$.
\end{itemize}

\section{Counting admissible triples} \label{sec:counting}

The main idea of this paper is to count the number $N^{(4)}(p)$ of Diophantine quadruples over $\F$, by counting the admissible triples $(Q_1, Q_2, Q_3)$.

Since one Diophantine quadruple $\{a,b,c,d\}$ corresponds to several admissible triples, we count each admissible triple with weight $w=w(Q_1,Q_2,Q_3)$ where $48/w$ (or $24/w$ if $\{-a,-b,-c,-d\}=\{a,b,c,d\}$) is equal to the number of admissible triples that correspond to the same Diophantine quadruple(s) as $(Q_1, Q_2, Q_3)$. More precisely, if some $Q_i$ or $Q_i \pm R$ has order $2$ then the weight is  $w=\frac{1}{2^4}$ (because one $t_{ij}$ will be equal to $0$), otherwise it is $w=\frac{1}{2^5}$. Note that if there are two $Q_i$ and $Q_j$ such that $2Q_i=\pm R$ and $2Q_j=\pm R$, then the corresponding Diophantine quadruple has the property that $\{-a,-b,-c,-d\}=\{a,b,c,d\}$ and we count such a triple with weight $w=\frac{1}{2^5}$ (unless the third element $Q_k$ or $Q_k\pm R$ has order $2$, in which case the weight is $w=\frac{1}{2^4}$). For $t \in \F\backslash \{0,1\}$, denote by
\[
W(t) = \frac{1}{24}\sum_{(Q_1, Q_2, Q_3)} w(Q_1, Q_2, Q_3),
\]
where the sum is over all admissible triples $(Q_1, Q_2, Q_3) \in E_t(\F)^3$. Thus $W(t)$ is equal to the number of Diophantine quadruples $\{a,b,c,d\}$ with $abcd=t$.
Also for the correct count, Diophantine quadruples corresponding to the singular fiber $\mathcal{D}_1$ will be counted separately, we denote their number by $W(1)$. For $t \in \F, t\ne 0, 1$, denote by $P(t) = \# E_t(\F)$.

For every $Q \in \E$, denote by $[Q]$ the set $\{ Q + kR, -Q+kR: k \in \{0, 1, 2, 3\}\}$. We call such a set the class of $Q$. Note that $\#[Q]=8$, unless $[Q]$ contains a point of order $2$ (different than $2R=(0,0)$) or a point $Q'$ such that $2Q'=\pm R$, in which case $\#[Q]=4$.

\begin{proposition}\label{prop:square} Let $t \ne 0,1$  in $\F$ be such that $E_t(\F)[2]\cong \Z/2\Z\times \Z/2\Z$.
\begin{itemize}
	\item [a)] Let $T\in E_t(\F)$ be a point of order two, $T \ne 2R$. Then
	$x(T)$ is a square if and only if $p \equiv 1 \pmod{4}$.
	\item [b)] Let $Q \in E_t(\F)$ satisfy $2Q=\pm R$, and let $P\in E_t(\overline{\F})$ be such that $2P=Q$. Then
	$x(Q)$ is a square if and only if the subgroup $\langle P \rangle\le E_t(\overline{\F})$ generated by $P$ is $\F$-rational.
  \item[c)] Let $T \in E_t(\F)$ satisfy $2T = \mathcal{O}$ ($T \ne 2R$), and let $P\in E_t(\overline{\F})$ be such that $2P=T$. Then $x(T)$ is a square if and only if $P^\sigma - P \in \{\mathcal{O}, 2R\}$, for all $\sigma \in \textrm{Gal}(\overline{\F}/\F)$.
\end{itemize}
\end{proposition}
\begin{proof}
\begin{itemize}[leftmargin=0.5 cm]
	\item [a)] The $x$-coordinates of the points of order two satisfy $x(x^2+(4-2t)x+t^2)=0$. In particular, $x(T)=t-2\pm2\sqrt{1-t}=-(\pm\sqrt{1-t}-1)^2$ is a square if and only if $\g{-1}=1$ (since $\g{1-t}=1$). Hence the claim follows.
	\item[b)] It follows from the explicit two-descent theory (see Theorem 1.1. in Chapter X of \cite{Sil}) that there is a bilinear pairing $$b:E_t(\F)/2E_t(\F) \times E_t[2] \rightarrow \Ff/{\Ff}^2$$ satisfying
	$$e_2(P^\sigma-P,2R)=\delta_{\F}(b(Q,2R))(\sigma) \textrm{ for every }\sigma \in \mathrm{Gal}(\overline{\F}/\F),$$
	where $e_2$ is the Weil pairing, and $\delta_{\F}$ is the connecting homomorphism for the Kummer sequence associated to the group variety $\mathbb{G}_m/\F$. Moreover, $b(Q,2R)\equiv x(Q) \pmod{{\Ff}^2}$. In particular, $x(Q)$ is square if and only if $e_2(P^\sigma-P, 2R)=1$ for all $\sigma$, or equivalently if and only if $P^\sigma-P\in \{\mathcal{O}, 2R \}$ for all $\sigma$ (since $e_2(2R,2R)=e_2(\mathcal{O},2R)=1$ and $e_2(T,2R)=e_2(T+2R,2R)=-1$ where $T\ne 2R$ is a point of order two). The claim follows since $4P=\pm R$.
	\item[c)] Same as in b), $x(T)$ is square if and only if $e_2(P^\sigma-P,2R)=1$ for all $\sigma$. Hence, the claim follows.
\end{itemize}
\end{proof}
\begin{remark}
Note that half of the points in $E_t(\F)$ will have a $x$-coordinate equal to a square.
\end{remark}

\noindent For the rest of the section fix $t \ne 0,1$ in $\F$. We calculate $W(t)$ in the following three cases.
\subsection{$R$ is not divisible by $2$ in $\E$}
\begin{itemize}
\item[a)]
In the case where $x(R)\ne \square$, we can count triples $(Q_1,Q_2,Q_3)$ by first choosing three different classes $([Q_1], [Q_2], [Q_3])$, and then choosing all possible elements from these classes. Half of these triples will be admissible since for every $P \in \E$ precisely one of $x(P)$ and $x(P+R)$ is a square since by two-isogeny descent homomorphism we have that $x(P+R)\equiv x(P)x(R) \pmod{{\Ff}^2}$(see the proof of Lemma \ref{lem:rat}). We consider two cases:
\begin{itemize}[leftmargin=0.8 cm]
	\item [i)] $E_t(\F)[2]\cong \Z/2\Z$ ($2R$ is the only point of order two and $x(R)\ne \square$)\\
	All the classes in $\E$ (note that we don't consider class $[R]$ which has order $4$) have eight elements and weight $w=\frac{1}{2^5}$. Denote by $b$ the total number of classes. Then $b=\frac{P(t)-4}{8}$ and
	\begin{equation*}
	24 W(t)=w\cdot b(b-1)(b-2)2^8=\ds\frac{(P(t)-20)(P(t)-12)(P(t)-4)}{64}.
	\end{equation*}
	\item[ii)] $E_t(\F)[2]\cong \Z/2\Z\times \Z/2\Z$ and $x(R)\ne\square$\\ Let $T\in E_t(\F)$ be a point of order two different than $2R$.
	The class $[T]$ contains four elements (and triples containing its elements have weight $w=\frac{1}{2^4}$), while  other $b=\frac{P(t)-8}{8}$ classes different from $[T]$ and $[R]$ contain eight elements. Hence,
\begin{equation*}	
24 W(t)= \frac{1}{2^5} b(b-1)(b-2)2^8+\frac{1}{2^4}3 \cdot b(b-1)2^{7}
    = \ds \frac{P(t)(P(t)-8)(P(t)-16)}{64}.
\end{equation*}
\end{itemize}
\item[b)]
In the case where $x(R)=\square$, we consider two cases:
\begin{itemize}[leftmargin = 0.8cm]
	\item [i)] $E_t(\F)[2]\cong \Z/2\Z\times \Z/2\Z$ and  $p\equiv 1 \pmod{4}$\\
	Proposition \ref{prop:square} implies that $x(T)=\square$, hence there are $b_1=\frac{P(t)-16}{16}$ eight-element classes $[Q_i]$ for which $x(Q_i)$ is a square, and $b_2=\frac{P(t)}{16}$ eight-element classes $[Q_i]$ for which $x(Q_i)$ is not a square. Hence,
\begin{eqnarray*}
24 W(t)&=&\frac{1}{2^5}b_1(b_1-1)(b_1-2)2^9+\frac{3}{2^5}b_1 b_2 (b_2-1) 2^9+\frac{3}{2^4}(b_1(b_1-1)+b_2(b_2-1))2^8\\
&=& \ds \frac{P(t)(P(t)-8)(P(t)-16)}{64}.
\end{eqnarray*}
Note that in this case $W(t)$ is equal to the $W(t)$ from the $b)$.
	\item [ii)] $E_t[2](\F)\cong \Z/2\Z\times \Z/2\Z$ and $p\equiv 3 \pmod{4}$\\
	Proposition \ref{prop:square} implies that $x(T)\ne \square$, hence there are $b=\frac{P(t)-8}{16}$ classes $[Q_i]$ for which $x(Q_i)$ is a square, and $b$ classes $[Q_i]$ for which $x(Q_i)$ is not a square. Hence
\begin{eqnarray*}
24W(t)&=& \frac{1}{2^5}(b(b-1)(b-2)2^9+3b(b-1)b \cdot 2^9)+\frac{6}{2^4}(b\cdot b)2^8\\
&=& \frac{1}{64} (P(t) - 8) (P(t)^2 - 16P(t) + 192).
\end{eqnarray*}
\end{itemize}
\end{itemize}

\subsection{$2Q=R$ and $x(Q)\ne \square$}

Since $2Q=R$ for some $Q\in E_t(\F)$, we have that $x(P) \equiv x(P+R) \pmod{{\Ff}^{2}}$ for all $P \in E$, so for any class $[Q_1]$ in $\E$ the $x$-coordinates of the points in $[Q_1]$ are either all squares or non-squares. We consider two cases:

\begin{itemize}[leftmargin=0.8cm]
	\item [a)] $E_t(\F)[2]\cong \Z/2\Z$ ($2R$ is the only point of order two)\\
	The class $[Q]$ contains four points, while the other $b=\frac{P(l)-8}{8}$ classes contain eight points. All triples have weight $w=\frac{1}{2^5}$. There are precisely $\frac{b}{2}$ classes $[Q_1]$ for which $x(Q_1)$ is equal to a square (since $x(Q)$ is not a square, and half of the points in $E_t(\F)$ have a $x$-coordinate which is a square). Hence
\begin{eqnarray*}
24 W(t)&=&\frac{1}{2^5}\left(\frac{b}{2}\left(\frac{b}{2}-1\right)\left(\frac{b}{2}-2\right)2^9+3\left(\frac{b}{2}\right)^2\left(\frac{b}{2}-1\right)2^9+3!\left(\frac{b}{2}\right)^2 2^8 \right)\\
&=& \ds \frac{(P(t)-8)(P(t)^2-28P(t)+288)}{64}.
\end{eqnarray*}
	
	\item [b)] $E_t(\F)[2]\cong \Z/2\Z\times \Z/2\Z$ ($T\in E_t(\F)$ is a point of order two different than $2R$)\\
	There are three classes with four elements: $[Q]$, $[T]$, and $[T+Q]$.
	\begin{enumerate}
		\item [i)] $p\equiv 1 \pmod{4}$\\
		In this case Proposition \ref{prop:square} implies that $x(T)$ is a square, and $x(Q+T)$ is not a square. There are $b=\frac{P(t)-16}{8}$ classes with eight elements, half of which have $x$-coordinate equal to a square. We calculate
\begin{eqnarray*}
24W(t)&=& \frac{1}{2^5}\left(\frac{b}{2}\left(\frac{b}{2}-1\right)\left(\frac{b}{2}-2\right)2^9+3\left(\frac{b}{2}\right)^2\left(\frac{b}{2}-1\right)2^9 \right)\\&+&\frac{2\cdot 3}{2^4}\left(\frac{b}{2}\right)\left(\frac{b}{2}-1\right) 2^8
+ \frac{2}{2^5}\left( 3!\left(\frac{b}{2}\right)^2\right)2^8+\frac{3!}{2^4}\frac{b}{2}2^7\\&+&\frac{3!}{2^4}\frac{b}{2}2^7+\frac{3!}{2^5}\frac{b}{2}2^7
+ \frac{3!}{2^4}2^6\\
&=&\frac{P(t)(P(t)^2-24P(t)+224)}{64}.
\end{eqnarray*}		
		\item [ii)] $p \equiv 3 \pmod{4}$\\
		In this case Proposition \ref{prop:square} implies that $x(T)$ is not a square, and $x(Q+T)$ is a square. There are $b=\frac{P(t)-16}{8}$ classes with eight elements, half of which have $x$-coordinate equal to a square. We calculate $W(t)$ as in i).
\begin{eqnarray*}
24W(t)&=& \frac{1}{2^5}\left(\frac{b}{2}\left(\frac{b}{2}-1\right)\left(\frac{b}{2}-2\right)2^9+3\left(\frac{b}{2}\right)^2\left(\frac{b}{2}-1\right)2^9 \right)\\&+&\frac{3!}{2^4}\left(\frac{b}{2}\right)^2 2^8
+ \frac{1}{2^5}\left( 3!\left(\frac{b}{2}\right)^2 + 3 \cdot 2\cdot\frac{b}{2}\left(\frac{b}{2}-1\right)\right)2^8+\frac{3!}{2^4}\frac{b}{2}2^7\\&+&\frac{3!}{2^4}\frac{b}{2}2^7+\frac{3!}{2^5}\frac{b}{2}2^7
+ \frac{3!}{2^4}2^6\\
&=&\frac{P(t)^3-24P(t)^2+416 P(t)-3072}{64}.
\end{eqnarray*}				
	\end{enumerate}
\end{itemize}

\subsection{$2Q=R$ and $x(Q) = \square$}
We consider two cases.
\begin{itemize}[leftmargin=0.8cm]
	\item [a)]  $E_t(\F)\cong \Z/2\Z$ ($2R$ is the only point of order two)\\
	The class $[Q]$ contains four points, while the other $b=\frac{P(t)-8}{8}$ classes contain eight points. All triples have weight $w=\frac{1}{2^5}$. There are $b_1=\frac{b-1}{2}$ classes $[Q_1]$ with $x(Q_1)=\square$, and $b_2=\frac{b+1}{2}$ classes $[Q_1]$ with $x(Q_1) \ne \square$. We have
	\begin{eqnarray*}
	24W(t)&=& \frac{1}{2^5}\left( b_1(b_1-1)(b_1-2)+3b_1 b_2(b_2-1) \right)2^9+\frac{1}{2^5}\left(3 b_1(b_1-1)+3 b_2(b_2-1) \right)2^8\\
	&=&\ds \frac{(P(t)-16)(P(t)^2-20P(t)+192)}{64}.
	\end{eqnarray*}
	\item [b)]  $E_t(\F)\cong \Z/2\Z\times \Z/2\Z$ ($T\in E_t(\F)$ is a point of order two different than $2R$)\\
There are three classes with four elements: $[Q]$, $[T]$, and $[T+Q]$, and $b=\frac{P(t)-16}{8}$ classes with eight elements.

\begin{enumerate}
	\item [i)] $p \equiv 1 \pmod{4}$\\
	Proposition \ref{prop:square} implies that $x(T)$ and $x(T+Q)$ are both  squares. Then $b_1=\frac{b-2}{2}$  and $b_2=\frac{b+2}{2}$.
	\begin{eqnarray*}
24 W(t) &=& \frac{1}{2^5}\left( b_1(b_1-1)(b_1-2)+3b_1 b_2(b_2-1) \right)2^9+\left(2+2\right)\frac{1}{2^5}\left(3 b_1(b_1-1)+3 b_2(b_2-1) \right)2^8\\
&+&\frac{2}{2^4}3! b_1 2^7+\frac{1}{2^5}3!b_1 2^7 + \frac{1}{2^4}3!2^6\\
&=& \frac{(P(t)-16)(P(t)^2-8P(t)+96)}{64}.
\end{eqnarray*}
	\item [ii)] $p \equiv 3 \pmod{4}$\\
	Proposition \ref{prop:square} implies that $x(T)$ and $x(T+Q)$ are not squares. Then $b_1=b_2=\frac{b}{2}$, and we calculate
	\begin{eqnarray*}
24 W(t) &=& \frac{1}{2^5}\left( b_1(b_1-1)(b_1-2)+3b_1 b_2(b_2-1) \right)2^9+\frac{1}{2^5}\left(3 b_1(b_1-1)+3 b_2(b_2-1) \right)2^8\\
&+& (1+2)\frac{1}{2^5}\left( 3!b_1 b_2\right)2^8+\frac{1}{2^4}3! b_1 2^7+\frac{1}{2^4}3! b_2 2^7+\frac{1}{2^5}3! b_1 2^7+\frac{1}{2^4}3!2^6\\
&=& \frac{P(t)^3-24P(t)^2+416P(t)-3072}{64}.
	\end{eqnarray*}
\end{enumerate}
\end{itemize}
\subsection{Putting everything together}
For the fixed prime $p$ we define the following sets:
\begin{eqnarray*}
T_0&=&\left\{ t \in \F^\times\setminus\{1\}: E_t[2](\F) = \frac{\Z}{2\Z} \times  \frac{\Z}{2\Z} \textrm{ and } x(R)=\square\right\}\\
T_1&=&\left\{t \in \F^\times\setminus\{1\}:  E_t[2](\F) = \frac{\Z}{2\Z} \times  \frac{\Z}{2\Z}  \right\}\\
T_2&=&\left\{t \in \F^\times\setminus\{1\}: E_t[2](\F) = \frac{\Z}{2\Z}  \times  \frac{\Z}{2\Z}  \textrm{ and }2Q = R \textrm { for some } Q \in E_t(\F) \right\}\\
T_3&=&\left\{t \in \F^\times\setminus\{1\}: 2Q = R \textrm { for some } Q \in E_t(\F) \right\}\\
T_4&=&\left\{ t \in T_3: \langle S \rangle \textrm{ is } \F \textrm{-rational, where } 2S\in E_t(\F) \textrm{ and } 4S=R  \textrm{ for some }S \in E_t \right\}\\
T_5&=&\left\{t \in T_4:  E_t[2](\F) = \frac{\Z}{2\Z} \times  \frac{\Z}{2\Z}\right\}.
\end{eqnarray*}

\begin{remark}
 Note that it follows from Proposition \ref{prop:square} that
$$T_4 = \{ t \in \F^\times\setminus\{1\}: 2Q = R \textrm{ and } x(Q)=\square \textrm { for some } Q \in E_t(\F)\}.$$ Also, if $p \equiv 3 \pmod{4}$, then $T_2=T_5$.
\end{remark}

We have the following proposition.

\begin{proposition}\label{prop:main}
\begin{enumerate}
	\item [a)] If $p \equiv 1 \pmod{4}$, then
\begin{eqnarray*}	
24 \sum_{t\ne 0,1} W(t)&=&\sum_{t\ne 0,1}\left( \frac{1}{64}P(t)^3 -\frac{9}{16} P(t)^2 + \frac{23}{4} P(t) -15\right)\\  &+&  \sum_{t \in T_1}\left( \frac{3}{16}P(t)^2-\frac{15}{4}P(t)+ 15\right)-\sum_{t \in T_2}\left(\frac{3}{4}P(t)-21 \right)\\ &+& \sum_{t \in T_3}\left(\frac{9}{4}P(t)-21\right)-12\sum_{t \in T_4}1-12\sum_{t \in T_5}1.
\end{eqnarray*}
	\item [b)] If $p \equiv 3 \pmod{4}$ then
	\begin{eqnarray*}	
24 \sum_{t\ne 0,1} W(t)&=&\sum_{t\ne 0,1}\left( \frac{1}{64}P(t)^3 -\frac{9}{16} P(t)^2 + \frac{23}{4} P(t) -15\right)\\  &+&  \sum_{t \in T_1}\left( \frac{3}{16}P(t)^2-\frac{15}{4}P(t)+ 15\right)-\sum_{t \in T_2}\left(\frac{3}{4}P(t)+3 \right)\\ &+& \sum_{t \in T_3}\left(\frac{9}{4}P(t)-21\right)+\sum_{t\in T_0}\left(3P(t)-24 \right)-12\sum_{t \in T_4}1+12\sum_{t \in T_5}1.
\end{eqnarray*}

\end{enumerate}
\end{proposition}
\subsection{Calculating $W(1)$}
The curve $\mathcal{D}_1: (x^2-1)(y^2-1)=1$ is birationally equivalent to the genus zero curve $E_1: S^2=T(T+1)^2$. Analysis similar (but easier) to the one in Section \ref{sec:cor} yields the following proposition.
\begin{proposition}\label{prop:W1}
\[
24 \cdot W(1)=
\begin{cases}
	\frac{(p-9)(p^2-18p+113)}{32}, \textrm{ if } p \equiv 1 \pmod{8}, \\
	\frac{(p-3)(p-11)(p-19)}{32}, \textrm{ if } p \equiv 3 \pmod{8}, \\
	\frac{(p-5)(p-9)(p-13)}{32}, \textrm{ if } p \equiv 5 \pmod{8},\\
	\frac{(p-7)(p-11)(p-15)}{32}, \textrm{ if } p \equiv 7 \pmod{8}. \\
\end{cases}
\]
\end{proposition}

\section{Families of universal elliptic curves and $\ell$-adic representations}
\label{sec:universal}
\subsection{Modular curves and cusps}\label{sec:curves}
For $M, N \ge 1$, $M|N$, we denote by $Y(M,N)$ the quotient of the upper half plane by the congruence subgroup $\Gamma_1(N)\cap \Gamma^0(M)$. Here $\Gamma^0(M)=\left\{{\sm a b c d} \in \SL_2(\Z): {\sm a b c d} \equiv {\sm * 0 * * } \pmod{M} \right\}$. As a modular curve (irreducible, connected and defined over $\Q(\zeta_M)$) $Y(M,N)$ parametrizes elliptic curves $E$ together with the points $P$ and $Q$ of order $M$ and $N$, such that $P$ and $Q$ generate subgroup of order $MN$ and the Weil pairing $e_{N/M}$ between the points $P$ and $\frac{N}{M}Q$ is equal to the fixed primitive $M$-th root of unity, i.e. $e_{N/M}(P,\frac{N}{M}Q)=e^{2\pi i /M}$. We denote by $X(M,N)$ the compactification of $Y(M,N)$. For more information on modular curves $Y(M,N)$ see Section 2 in \cite{Kato}.

In Section \ref{sec:results}, we will need to know the number of $\F$-rational cusps on modular curves $X_1(8)_{\F}, X(2,4)_{\F}, X(2,8)_{\F}$ and $X(4,8)_{\F}$. Following \cite[Section 2]{BN}, we briefly explain how to calculate the field of definition of cusps on $X(M,N)$.

Let $r$ be a divisor of $N$. The cusps of $X(M,N)$ represented by the points $(a:b)\in \mathbb{P}^1(\Q)$, where $a,b$ are co-prime integers with $\gcd(b,N)=r$, all have the same field of definition, $\Q(\zeta_M) \le F_r \le \Q(\zeta_N)$. If we canonically identify $\textrm{Gal}(\Q(\zeta_N)/\Q)$ with $(\Z/N\Z)^\times$, then $F_r$ is the fixed field of the group $H_r$ acting on $\Q(\zeta_N)$, where
$H_r= H_r^0 :=\{ s \in (\Z/N\Z)^\times: s \equiv 1 \pmod{\textrm{lcm}(M, N/r)}\},$ if $\gcd(Mr,N) > 2$, and $H_r = H_r^0 \cdot \{ \pm 1\}$ otherwise.

It follows immediately that all four cusps of $X(2,4)$ are $\Q$-rational (i.e. $c(2,4)=4$), and that the number $c(2,8)$ of $\F$-rational cusps of $X(2,8)_{\F}$ is equal to
\[
c(2,8)=
\begin{cases}
	10, \textrm{ if } p \equiv 1 \pmod{8},\\
	4, \textrm{ if } p \equiv 3 \pmod{8},\\
	6, \textrm{ if } p \equiv 5 \pmod{8},\\
	8, \textrm{ if } p \equiv 7 \pmod{8}.\\
\end{cases}
\]

Moreover, the curve $X(4,8)_{\Q(i)}$ has eight cusps defined over $\Q(i)$ and eight cusps defined over $\Q(\zeta_8)$, hence the number of $\F$-rational cusps is equal to $$
c(4,8)=
\begin{cases}
	16, \textrm{ if } p \equiv 1 \pmod{8},\\
	8, \textrm{ if } p \equiv 5 \pmod{8}.\\
\end{cases}
$$
The number of $\F$-rational cusps on modular curve $X_1(8)_{\F}$ is equal to
\[
c(8)=
\begin{cases}
	6, \textrm{ if } p \equiv 1,7 \pmod{8},\\
	4, \textrm{ if } p \equiv 3,5 \pmod{8}.\\
\end{cases}
\]

The curves $X_1(4), X(2,4), X_1(8)$ and $X(2,8)$ have genus zero, while the curve $X(4,8)$ has genus one.

\subsection{Modular forms} \label{sec:forms}

Here we collect some facts about the spaces of modular forms related to the modular curves from the previous subsection. They can be checked using Sage \cite{SAGE} and LMFDB database \cite{LMFDB}.

\begin{proposition}\label{prop:forms} Denote by $T_p$ the $p$-th Hecke operator acting on the space of cusp forms $S_3(\Gamma_1(8) \cap \Gamma^0(2))$. We have
\begin{itemize}
	\item [a)] $\dim S_3(\Gamma_1(4))=\dim S_4(\Gamma_1(4))=0  \textrm{ and } S_5(\Gamma_1(4))= \mathbb{C} \cdot f_5(\tau)$,
	\item [b)] $\dim S_3(\Gamma_1(8) \cap \Gamma^0(2))=3 \textrm{ and } \textrm{Trace}(T_p)=2b(p)+c(p),$
	\item [c)] $S_3(\Gamma_1(8))=\mathbb{C}\cdot f_2(\tau),$
	\item [d)] $S_3(\Gamma_1(4) \cap \Gamma^0(2))=0$ and $S_4(\Gamma_1(4) \cap \Gamma^0(2))=\mathbb{C}\cdot f_4(\tau)$.
\end{itemize}
\end{proposition}
Modular forms $f_1(\tau), f_2(\tau), f_3(\tau)$ and $f_5(\tau)$ are CM forms, and their Fourier coefficients are given in the following proposition. For some standard facts about CM modular forms see \cite[p.9]{ONO}.
\begin{proposition} Let $p$ be an odd prime and $q=e^{2\pi i \tau}$. We have
\begin{itemize}
\item[a)] $f_1(\tau)=\eta^2(4\tau)\eta^2(8\tau)=q\prod_{n=1}^\infty(1-q^{4n})^2(1-q^{8n})^2$, and
\[
a(p)=\begin{cases}
\pm 2x,\textrm{ if } p \equiv 1\pmod{4} \textrm{ and } p = x^2 + 4y^2\\0, \textrm{ if } p \equiv 3 \pmod{4},
\end{cases}
\]
\item[b)] $f_2(\tau)=\eta^2(\tau)\eta(2\tau)\eta(4\tau)\eta^2(8\tau)=q\prod_{n=1}^\infty (1-q^n)^2(1-q^{2n})(1-q^{4n})(1-q^{8n})^2$, and
\[
b(p)=\begin{cases}
  2(x^2-2y^2),\textrm{ if } p \equiv 1, 3\pmod{8} \textrm{ and } p = x^2 + 2y^2\\0, \textrm{ if } p \equiv 5, 7 \pmod{8},
\end{cases}
\]
\item[c)] $f_3(\tau)=\eta^6(4\tau)=q\prod_{n=1}^\infty (1-q^{4n})^6$, and
\[
c(p)=\begin{cases}
\pm 2(x^2-4y^2),\textrm{ if } p \equiv 1\pmod{4} \textrm{ and } p = x^2 + 4y^2\\0, \textrm{ if } p \equiv 3 \pmod{4},
\end{cases}
\]
\item[d)] $f_4(\tau)=\eta^4(2\tau)\eta^4(4\tau)=q\prod_{n=1}^\infty (1-q^{2n})^4 (1-q^{4n})^4$,
\item[e)] $f_5(\tau) = \eta^4(\tau)\eta^2(2\tau)\eta^4(4\tau)=q\prod_{n=1}^\infty (1-q^n)^4(1-q^{2n})^2(1-q^{4n})^4$, and
\[e(p)=\begin{cases}
2p^2-16 x^2 y^2,\textrm{ if } p \equiv 1\pmod{4} \textrm{ and } p = x^2 + y^2\\0, \textrm{ if } p \equiv 3 \pmod{4}.
\end{cases}
\]
\end{itemize}
\end{proposition}

\subsection{Families of universal elliptic curves}\label{sec:families}
Let $E^1, E^2, E^3, E^4$ and $E^5$ be elliptic surfaces fibered over the modular curves $X_1(4)$, $X(2,4)$, $X_1(8)$, $X(2,8)$ and $X(4)$ defined by affine equations (given with the sections of the corresponding orders):
\begin{align*}
E^1: Y^2&=X(X^2-2(t_1-2)X+t_1^2), \quad P_1=[t_1, 2t_1];4 P_1 = \mathcal{O}\\
E^2: Y^2&= X(X + t_2^2 - 2t_2 + 1)(X + t_2^2 + 2t_2 + 1),\\ P_2&=[1-t_2^2, 2(1-t_2^2)],\quad T_2=[- t_2^2 + 2t_2 - 1,0]; 4P_2=2 T_2= \mathcal{O} \\
E^3: Y^2&=X\left(X^2-2(t_3^4-2t_3^2-1)+(t_3-1)^4(t_3+1)^4\right), \\ Q_3&=[(t_3-1)(t_3+1)^3,2 t_3 (t_3-1)(t_3+1)^3]; 8 Q_3 = \mathcal{O}\\
E^4: Y^2&=X \left(X+\frac{64 t_4^4}{(t_4^2+1)^4}\right)\left(X+\frac{4(t_4-1)^4(t_4+1)^4}{(t_4^2+1)^4}\right),\\ Q_4&=\left[\frac{-16 t_4(t_4-1)(t_4+1)^3}{(t_4^2+1)^4}, \frac{32 t_4(t_4-1)(t_4+1)^3(t_4^2-2 t_4-1)}{(t_4^2+1)^5}\right], T_4=\left[-\frac{64 t_4^4}{(t_4^2+1)^4},0\right];\\8Q_4&=2T_4=\mathcal{O}\\
E^5: Y^2&=X(X + t_5^4 - 2 t_5^2 + 1)(X + t_5^4 + 2 t_5^2 + 1),\\
Q_5&=[1-t_5^4,2-2t_5^4], \quad Q^{'}_5=[-t_5^4 + 2 i t_5^3 + 2 t_5^2 - 2 i t_5 - 1,2 t_5^5 - 4 i t_5^4 - 4 t_5^3 + 4 i t_5^2 + 2 t_5];\\4Q_5&=4Q_5^{'}=\mathcal{O}\\
\end{align*}
together with the maps
\begin{align*}
h_1&:E^1\rightarrow X_1(4) \quad (X,Y,t_1) \mapsto t_1,\\
h_2&:E^2\rightarrow X(2,4)\quad (X,Y,t_2) \mapsto t_2,\\
h_3&:E^3\rightarrow X_1(8)\quad (X,Y,t_3) \mapsto t_3,\\
h_4&:E^4\rightarrow X(2,8)\quad (X,Y,t_4) \mapsto t_4,\\
h_5&:E^5\rightarrow X(4)\quad (X,Y,t_5) \mapsto t_5.\\
\end{align*}

Here we identify modular curves $X_1(4)$, $X(2,4)$, $X_1(8)$, $X(2,8)$ and $X(4)$ with $\mathbb{P}^1$ using parameters $t_1$, $t_2$, $t_3$, $t_4$ and $t_5$.

We have the natural maps
\begin{align*}
g_2&:X(2,4)\rightarrow X_1(4), \quad (E, T_2, P_2) \mapsto (E,P_2), \quad t_1 = 1-t_2^2\\
g_3&:X_1(8)\rightarrow X_1(4), \quad (E,Q_3)) \mapsto (E, 2Q_3),\quad t_1=(t_3^2-1)^2\\
g_4&:X(2,8)\rightarrow X_1(4), \quad (E, T_4, Q_4) \mapsto (E,2Q_4), \quad t_1 = \frac{16 t_4^2 (t_4-1)^2 (t_4+1)^2}{(t_4^2+1)^4} ,\\
g_5&:X(4)\rightarrow X_1(4), \quad (E, Q_5, Q^{'}_5) \mapsto (E,Q_5), \quad t_1 = 1-t_5^4 .\\
\end{align*}

Elliptic surfaces $E^1, E^2, E^3, E^4$ and $E^5$ are universal elliptic curves over the modular curves $X_1(4), X(2,4), X_1(8), X(2,8)$ and $X(4)$ respectively (for the universality, it is enough to check that for each $i$ the degree of $j$-invariant $j(E^i)$ is equal to the index of the corresponding subgroup in $\SL_2(\Z)$). Note that $X(4)$ is defined over $\Q(i)$.
\subsection{Model for $X(4,8)$}

In this section we calculate $\#X(4,8)(\F)$, where $p\equiv 1 \pmod{4}$. Denote by $t_5'$ and $t_4'$ pullbacks of function $t_4$ and $t_5$ on $X(2,8)$ and $X(4)$ along the natural maps $X(4,8) \rightarrow X(2,8)$, $(E,P,Q) \mapsto (E,2P,Q)$ and $X(4,8) \rightarrow X(4)$, $(E,P,Q) \mapsto (E,P,2Q)$. Then we have
\[
1-t_5'^4=\frac{16 t_4'^2 (t_4'-1)^2 (t_4'+1)^2}{(t_4'^2+1)^4},
\]
which implies
\[
(t_5'(t_4'^2+1))^2 = \pm (t_4'^2-2t_4'-1)(t_4'^2+2t_4'-1).
\]
The two genus one curves $y^2=\pm (x^2-2x-1)(x^2+2x-1)$ are isomorphic to the conductor $32$ elliptic curve $C: y^2=x^3-x$, hence, over $\Q(i)$ the modular curve $X(4,8)$ is isomorphic to $C$ (since $X(4,8)$ is connected). For a different proof see Lemma 13 in \cite{NAJ}.

 Therefore, for prime a $p\equiv 1 \pmod{4}$ (which splits in $\Q(i)$) we have that
\[
\#X(4,8)(\F) = \# C(\F) = p+1-a(p),
\]
since the modular form $f_1(\tau)=\sum_{n=1}^\infty a(n) q^n$ corresponds to $C$ by the modularity theorem (as it is the only newform in $S_2(\Gamma_0(32))$).

\subsection{Compatible families of $\ell$-adic Galois representations of $\Gal$}To each of these elliptic surfaces and to every positive integer $k$, we can associate two compatible families of $\ell$-adic Galois representations of $\Gal$. To ease notation, we denote by $\Gamma_j$, for $j=1,2,3,4$, groups $\Gamma_1(4), \Gamma(2,4), \Gamma_1(8) \textrm{ and } \Gamma(2,8)$ respectively, and by $X(\Gamma_j)$ the corresponding modular curve.

We define the representation $\rho_{j,\ell}^k$ of $\Gal$ as follows. Let $X(\Gamma_j)^0$ be the complement in $X(\Gamma_j)$ of the cusps. Denote by $i$ the inclusion of $X(\Gamma_j)^0$ into
$X(\Gamma_j)$, and by $h_j': E^{j,0} \rightarrow X(\Gamma_j)^0$ the
restriction of elliptic surface $h_j$ to $X(\Gamma_j)^0$. For a prime $\ell$ we obtain a sheaf
\[
\mathcal{F}_\ell^j=R^1 h_j'{_*{\Q_{\ell}}}
\]
on $X(\Gamma_j)^0$, and also a sheaf $i_*\textrm{Sym}^{k}\mathcal{F}_\ell$ on
$X(\Gamma_j)$ (here $\Q_\ell$ is the constant sheaf on the elliptic surface $E^{j,0}$, and $R^1$ is derived functor). The action of $\Gal$ on the $\Q_\ell$-space
\[
W_{k,\ell}^j = H^1_{et}(X(\Gamma_j)\otimes \overline{\Q}, i_*\textrm{Sym}^{k}\mathcal{F}_\ell^j)
\]
defines $\ell$-adic representation $\rho_{j,\ell}^k$ which is pure of weight $k+1$.

The second family, $\tilde{\rho}_{j,\ell}^k$, is $\ell$-adic realization of the motive associated to the spaces of cusp forms $S_{k+2}(\Gamma_j)$.
 For the construction see \cite[Section 5]{SASD}.

Similarly as in \cite[Section 3]{ALL}, since the elliptic surface $E^j$ is the universal elliptic curve over the modular curve $X(\Gamma_j)$, we can argue that these two representations are isomorphic, i.e. $\rho_{j,\ell}^k \sim \tilde{\rho}_{j,\ell}^k$. In particular, we will frequently use the following proposition.

\begin{proposition}\label{prop:trag}
Let $k \ge 1$ be an integer and $j\in \{1, 2, 3, 4 \}$. Denote by $B$ the set of normalized Hecke eigenforms in $S_{k+2}(\Gamma_j)$.
For every odd prime $\ell\ne p$ we have
$$Trace(\rho_{j,\ell}^k(Frob_p))=\sum_{f \in B} a_f(p),$$
where $a_f(p)$ is the $p$-th Fourier coefficient of the eigenform $f$, and $Frob_p$ is a geometric Frobenius at $p$.
\end{proposition}

\subsection{Traces of Frobenius}

To simplify notation, denote $\mathcal{F}= R^1 h_j'{_*{\Q_{\ell}}}$ and $W=H^1_{et}(X(\Gamma_j)\otimes \overline{\Q}, i_*\mathcal{F})$. We denote by $Frob_p\in \Gal$ a geometric Frobenius at $p$. 
We have the following well known result.
\begin{theorem}
\label{thm:trace}
The following are true:
\begin{itemize}
\item[(1)] We have that
  \[
  Trace(Frob_p|W)=-\sum_{t\in X(\Gamma_j)(\mathbb{F}_p)} Trace(Frob_p|(i_*\mathcal{F})_t).
  \]
\item[(2)] If the fiber $E^j_t := h_j^{-1}(t)$ is smooth, then
  \[
  Trace(Frob_p|(i_*\mathcal{F})_t)=Trace(Frob_p|H^1(E^j_t,
  \Q_\ell))=p+1-\#E^j_t(\mathbb{F}_p).
  \]
	Furthermore,
	$$
	Trace(Frob_p|(i_*\textrm{Sym}^2\mathcal{F})_t)=Trace(Frob_p|(i_*\mathcal{F})_t)^2-p,
	$$
	and
	$$
	Trace(Frob_p|(i_*\textrm{Sym}^3\mathcal{F})_t)=Trace(Frob_p|(i_*\mathcal{F})_t)^3-2p \cdot Trace(Frob_p|(i_*\mathcal{F})_t).
  $$
\item[(3)] If the fiber $E^j_t$ is singular, then
  \begin{equation*}
    Trace(Frob_p|(i_*\mathcal{F})_t)=
    \begin{cases}
      1 & \text{if the fiber is split multiplicative}, \\
      -1 & \text{if the fiber is nonsplit multiplicative},\\
      0 & \text{if the fiber is additive}.
    \end{cases}		
	\end{equation*}	
		Furthermore, $Trace(Frob_p|(i_*Sym^2\mathcal{F})_t)=1$ if the fiber is multiplicative or potentially multiplicative (e.g. fiber $E_{\infty}^1$), and
	\begin{equation*}
    Trace(Frob_p|(i_*Sym^3\mathcal{F})_t)=
    \begin{cases}
      1 & \text{if the fiber is split multiplicative}, \\
		 -1 & \text{if the fiber is nonsplit multiplicative},\\
      0 & \text{if the fiber is potentially multiplicative}.
    \end{cases}		
	\end{equation*}

\begin{proof}
(1) is the consequence of the Lefschetz fixed point formula (\cite{Del}, Rapport 3.2).

For good $t$, $(i_*\mathcal{F})_t=H^1(E^j_t,\Q_\ell)$, hence the first formula in (2) follows. Note that if $\lambda_1$ and $\lambda_2$ are eigenvalues of $Frob_p$ acting on $(i_*\mathcal{F})_t$, then $\lambda_1^k, \lambda_1^{k-1}\lambda_2,\ldots, \lambda_1\lambda_2^{k-1}, \lambda_2^k$ are the eigenvalues of $Sym^k Frob_p$ acting on $Sym^k (i_* \mathcal{F})_t$. Since $(i_* Sym^k \mathcal{F})_t=Sym^k (i_* \mathcal{F})_t$, the second part of $(2)$ follows (note that determinant of $Frob_p$ is equal to $p$).

In order to calculate trace of $Frob_p$ at bad fibers, we follow 3.7 of \cite{Sch2}. If $t \in X(\Gamma_j)(\F)$, let $K$ be the function field of the connected component of $X(\Gamma_j)\otimes \F$ containing $t$. Let $v$ be the discrete valuation of $K$ corresponding to $t$, and $K_v$ the completion. Let $G_v$ be the absolute Galois group $\textrm{Gal}(K_v^\textrm{sep}/K_v)$, $I_v$ the inertia group, and $F_v$ a geometric Frobenius. Write $H_v = H^1(E^j\otimes K_v^\textrm{sep}, \Q_\ell)$. Then $H_v$ is a $G_v$-module and 
$$Trace(Frob_p|(i_*\textrm{Sym}^k\mathcal{F})_t ) = Trace(F_v| (Sym^k H_v)^{I_v}).$$ In the case of multiplicative reduction $H_v^{I_v}$ is one dimensional and $F_v$ acts on it as $1$ if the reduction is split multiplicative, and as $-1$ if the reduction is nonsplit multiplicative. If the reduction is additive, then $H_v^{I_v}=\{0\}$. The first formula in (3) follows (see Section 10 of Chapter IV in \cite{Sil2}). 

In our situation (see Lemma 5.2 and Exercises 5.11, 5.13 in \cite{Sil2}), inertia subgroup $I_v$ acts on $H_v$ as ${\sm \chi * 0  \chi}$, where $*$ is not identically zero, and $\chi$ is the character associated to $L_v=K_v\left(\sqrt{\frac{-c_4(E^j\otimes K_v)}{c_6(E^j\otimes K_v)}}\right)/K_v$. If reduction at $v$ is multiplicative this character is unramified (or trivial), and if the reduction is additive it is ramified. Denote by $Y$ a generator of $H_v^{I(K_v^{sep}/L_v)}$. Then a direct computation shows that $Y^2$ and $Y^3$ generate $(Sym^2 H_v)^{I(K_v^{sep}/L_v)}$ and $(Sym^3 H_v)^{I(K_v^{sep}/L_v)}$ respectively. If the reduction is multiplicative, then $I_v=I(K_v^{sep}/L_v)$. If the reduction is potentially multiplicative then $I_v$ acts on $Y$ as $\pm 1$. Hence $(Sym^3 H_v)^{I_v}=\{0\}$, and $(Sym^2 H_v)^{I_v}$ is generated by $Y^2$. The claim follows. 
 
\end{proof}
	
\end{itemize}
\end{theorem}

\section{Results}\label{sec:results}

\subsection{$X_1(4)$}

The universal elliptic curve $E^1$ over $X_1(4)$ has three singular fibers (over the cusps): additive $t=\infty$, split multiplicative $t=0$, and fiber $t=1$ which is split multiplicative if $p\equiv 1 \pmod{4}$ and nonsplit multiplicative if $p \equiv 3 \pmod{4}$. Moreover, the additive fiber $t=\infty$ becomes (split) multiplicative over quadratic extension of the base field. Denote by $\mathcal{F}= R^1 h_1'{_*{\Q_{\ell}}}$.

\begin{proposition} \label{prop:prva}
\begin{itemize}
	
	\item [a)]
	\[
		\sum_{t \ne 0,1} P(t) =
		\begin{cases}
		p^2-p \textrm{ if } p \equiv 1 \pmod{4},\\
		p^2-p-2 \textrm{ if } p \equiv 3 \pmod{4}.
		\end{cases}
	\]
	\item[b)]
	\[
		\ds\sum_{t \ne 0,1} P(t)^2 =
		\begin{cases}
			p^3+p^2-p-1, \textrm{ if } p\equiv 1 \pmod{4},\\
			p^3+p^2-5p-5, \textrm{ if } p\equiv 3 \pmod{4}.\\
		\end{cases}
	\]
	\item[c)]
	\[
		\ds\sum_{t \ne 0, 1} P(t)^3 =	
		\begin{cases}
			p^4+4p^3-4p-3+e(p), \textrm{ if } p \equiv 1 \pmod{4},\\
			p^4+4p^3-6p^2-20p-11+e(p), \textrm{ if } p \equiv 3 \pmod{4}.\\
		\end{cases}
	\]	
\end{itemize}
\end{proposition}
\begin{proof}
a) Parts (1) and (2) of Theorem \ref{thm:trace} imply that

\begin{eqnarray*}
Trace(Frob_p|W_{1,\ell}^1)&=&-\sum_{t \in \textrm{ cusps}}Trace(Frob_p|(i_*\mathcal{F})_t)-\sum_{t\ne 0,1}\left(p+1-P
(t)\right),\\
&=& \sum_{t \ne 0,1} P(t) -(p^2-p-2)-\sum_{t \in \textrm{ cusps}}Trace(Frob_p|(i_*\mathcal{F})_t).\\
\end{eqnarray*}

Since $\dim(S_3(\Gamma_1(4)))=0$ it follows that $Trace(Frob_p|W_{1,\ell}^1)=0$. Claim now follows from Theorem \ref{thm:trace} (3) and the description of reduction types of singular fibers for $p\equiv 1 \pmod{4}$ and $p \equiv 1 \pmod{4}$. \\
b) Since by Theorem \ref{thm:trace}(3) the trace at every singular fiber is $1$, we have that
\begin{eqnarray*}
Trace(Frob_p|W_{2,\ell}^1)&=&-\sum_{t \in \textrm{ cusps}}Trace(Frob_p|(i_*\textrm{Sym}^2\mathcal{F})_t)-\sum_{t\ne 0,1}\left(\left(p+1-P
(t)\right)^2-p\right),\\
&=& -3-\sum_{t \ne 0,1} P(t)^2-p^3+p^2+p+2+2(p+1)\sum_{t\ne 0,1}P(t).\\
\end{eqnarray*}
The claim follows from the part a) since $\dim(S_4(\Gamma_1(4)))=0$ (hence $Trace(Frob_p|W_{2,\ell}^1)=0$).\\
c) Theorem \ref{thm:trace} implies that
\begin{eqnarray*}
Trace(Frob_p|W_{3,\ell}^1)&=&-\sum_{t \in \textrm{ cusps}}Trace(Frob_p|(i_*\textrm{Sym}^3\mathcal{F})_t)-\sum_{t\ne 0,1}\left(\left(p+1-P
(t)\right)^3-2p(p+1-P(t))\right),\\
&=&\sum_{t \ne 0,1}P(t)^3-3(p+1)\sum_{t \ne 0,1} P(t)^2+(3(p+1)^2-2p)\sum_{t\ne 0,1}P(t)\\
&-&(p-2)(p+1)^3+2p(p-2)(p+1)-\sum_{t \in \textrm{ cusps}}Trace(Frob_p|(i_*\textrm{Sym}^3\mathcal{F})_t).\\
\end{eqnarray*}
The claim follows from the parts a) and b) and Proposition \ref{prop:forms}a) (hence $Trace(Frob_p|W_{3,\ell}^1)=e(p)$). Note that 
\[
		\sum_{t \in \textrm{cusps}} Trace(Frob_p|(i_*\textrm{Sym}^3\mathcal{F})_t) =
			\begin{cases}
			0+1+1=2, \textrm{ if } p \equiv 1 \pmod{4},\\
			0+1+(-1)=0, \textrm{ if } p \equiv 3 \pmod{4}.\\
			\end{cases}
\]

\end{proof}

\subsection{$X(2,8)$}  Universal elliptic curve $E^4$ over $X(2,8)$ has $10$ singular fibers: $t_4=\pm i$ (two cusps above $t_1=\infty$) and $t_4^2+2t_4-1=0$ and $t_4^2-2t_4-1=0$ (four cusps above $t_1=1$) which are split multiplicative if $p \equiv 1 \pmod{4}$ and nonsplit multiplicative otherwise, and split multiplicative $t_4=\pm 1, 0, \infty$(four cusps above $t_1=0$). Denote $\mathcal{F}= R^1 h_4'{_*{\Q_{\ell}}}$.

\begin{proposition} \label{prop:T2}
\begin{itemize}
  \item [a)]
	\[
		\sum_{t \in T_2} 1 =
			\begin{cases}
				\frac{p-9}{8}, \textrm{ if } p \equiv 1 \pmod{8},\\
				\frac{p-3}{8}, \textrm{ if } p \equiv 3 \pmod{8},\\
				\frac{p-5}{8}, \textrm{ if } p \equiv 5 \pmod{8},\\
				\frac{p-7}{8}, \textrm{ if } p \equiv 7 \pmod{8}.\\
			\end{cases}			
	\]
	\item [b)]
	\[
		\sum_{t \in T_2} P(t) =
			\begin{cases}
			\frac{p^2-8p+1+2b(p)+c(p)}{8}, \textrm{ if } p \equiv 1 \pmod{8},\\
			\frac{p^2-2p+1+2b(p)+c(p)}{8}, \textrm{ if } p \equiv 3 \pmod{8},\\
			\frac{p^2-4p+1+2b(p)+c(p)}{8}, \textrm{ if } p \equiv 5 \pmod{8},\\
			\frac{p^2-6p-7+2b(p)+c(p)}{8}, \textrm{ if } p \equiv 7 \pmod{8}.\\
			\end{cases}
	\]
	
\end{itemize}
\end{proposition}
\begin{proof}
a) Since $T_2$ is equal to the image of $\F$-points (which are not cusps) on $X(2,8)_{\F}$ under the natural map $g_4:X(2,8) \rightarrow X_1(4)$ of degree $8$, we have
$\sum_{t \in T_2} 1 = \frac{p+1-c(2,8)}{8},$ since $\#X(2,8)(\F)=p+1$. The claim follows.\\
b) From the definition of $T_2$ we have $\sum_{t \in T_2}P(t)=\ds\frac{1}{8}\sum_{t\in g_3(X(2,8)(\F)) \atop t\not \in  \textrm{ cusps}}P(t)$.
Theorem \ref{thm:trace} implies that (the sum is over $X(2,8)(\F)$)
\begin{eqnarray*}
Trace(Frob_p|W_{1,\ell}^4)&=&-\sum_{t \in \textrm{ cusps}}Trace(Frob_p|(i_*\mathcal{F})_t)-\sum_{t\notin  \textrm{cusps}}\left(p+1-P
(t)\right),\\
&=&-\sum_{t \in \textrm{ cusps}}Trace(Frob_p|(i_*\mathcal{F})_t) -(p+1)(p+1-c(2,8))+\sum_{t \notin \textrm{cusps}}P(t).\\
\end{eqnarray*}
It follows from Proposition \ref{prop:forms}b) that $Trace(Frob_p|W_{1,\ell}^4)=2b(p)+c(p)$. The claim follows since Theorem \ref{thm:trace}(3) implies
\[
		\sum_{t \in \textrm{cusps}} Trace(Frob_p|(i_*\mathcal{F})_t) =
			\begin{cases}
			10, \textrm{ if } p \equiv 1 \pmod{8},\\
			4, \textrm{ if } p \equiv 3 \pmod{8},\\
			6, \textrm{ if } p \equiv 5 \pmod{8},\\
			0, \textrm{ if } p \equiv 7 \pmod{8}.\\
			\end{cases}
\]

\end{proof}

\subsection{$X_1(8)$} Universal elliptic curve $E^3$ over $X_1(8)$ has  $6$ singular fibers: split multiplicative $t_3=\infty$ and $t_3=\pm 1$ (two cusps above $t_1=0$) and $t_3=0, \pm\sqrt{2}$ (three cusps above $t_1=1$) which are split multiplicative if $p \equiv 1 \pmod{4}$ and nonsplit multiplicative otherwise. Denote $\mathcal{F}= R^1 h_3'{_*{\Q_{\ell}}}$.

\begin{proposition}\label{prop:T3}
\begin{itemize}
	\item [a)]
	\[
		\sum_{t \in T_3} 1 =
			\begin{cases}
				\frac{3p-11}{8}, \textrm{ if } p \equiv 1 \pmod{8},\\
				\frac{3p-9}{8}, \textrm{ if } p \equiv 3 \pmod{8},\\
				\frac{3p-7}{8}, \textrm{ if } p \equiv 5 \pmod{8},\\
				\frac{3p-13}{8}, \textrm{ if } p \equiv 7 \pmod{8}.\\
			\end{cases}	
	\]
	\item [b)]
	\[
		\sum_{t \in T_3} P(t) =
			\begin{cases}
			\frac{3p^2-8p+2b(p)-c(p)+3}{8}, \textrm{ if } p \equiv 1 \pmod{8},\\
			\frac{3p^2-6p+2b(p)-c(p)-5}{8}, \textrm{ if } p \equiv 3 \pmod{8},\\
			\frac{3p^2-4p+2b(p)-c(p)+3}{8}, \textrm{ if } p \equiv 5 \pmod{8},\\
			\frac{3p^2-10p+2b(p)-c(p)-13}{8}, \textrm{ if } p \equiv 7 \pmod{8}.\\
			\end{cases}
	\]
\end{itemize}
\end{proposition}
\begin{proof}
a) By definition, $T_3$ is equal to the image of $\F$-points (which are not cusps) on $X_1(8)$ under the natural map $g_3: X_1(8) \rightarrow X_1(4)$ of degree $4$. Since we have
$p+1-c(8)=4\sum_{t \in T_2}1+2 \sum_{t \in T_3}1,$ the claim follows.\\
b) From the definition of $T_3$ it follows that $$4\sum_{t \in T_2}P(t)+2\sum_{t \in T_3\setminus T_2} P(t) = \sum_{t \in g_2(X_1(8)(\F))\atop t \notin \textrm{cusps}}P(t),$$
hence $\sum_{t \in T_3}P(t)=\ds\frac{1}{2}\sum_{t \in g_2(X_1(8)(\F))\atop t \notin \textrm{cusps}}P(t)-\sum_{t \in T_2}P(t)$. Theorem \ref{thm:trace} implies that
\begin{eqnarray*}
Trace(Frob_p|W_{1,\ell}^3)&=&-\sum_{t \in \textrm{ cusps}}Trace(Frob_p|(i_*\mathcal{F})_t)-\sum_{t\notin  \textrm{cusps}}\left(p+1-P
(t)\right),\\
&=&-\sum_{t \in \textrm{cusps}}Trace(Frob_p|(i_*\mathcal{F})_t) -(p+1)(p+1-c(8))+\sum_{t \notin \textrm{cusps}}P(t),\\
\end{eqnarray*}
where the sums are over $X_1(8)(\F)$.
It follows from Proposition \ref{prop:forms}c) that $Tr(Frob_p|W_{1,\ell}^3)=b(p)$, and the claim follows. Note that Theorem \ref{thm:trace} implies
\[
		\sum_{t \in \textrm{cusps}} Tr(Frob_p|(i_*\mathcal{F})_t) =
			\begin{cases}
			6, \textrm{ if } p \equiv 1 \pmod{8},\\
			2, \textrm{ if } p \equiv 3 \pmod{8},\\
			4, \textrm{ if } p \equiv 5 \pmod{8},\\
			0, \textrm{ if } p \equiv 7 \pmod{8}.\\
			\end{cases}
\]
\end{proof}

\subsection{$X(2,4)$}  Universal elliptic curve $E^2$ over $X(2,4)$ has  $4$ singular fibers:  $t_2=0$ (the cusp above $t_1=1$) and $t_2=\infty$ which are split multiplicative if $p \equiv 1 \pmod{4}$ and nonsplit multiplicative otherwise, and split multiplicative $t_2=\pm 1$ (two cusps above $t_1=0$). Denote $\mathcal{F}= R^1 h_2'{_*{\Q_{\ell}}}$.

\begin{proposition}
\begin{itemize}
	\item [a)]
	\[
		\sum_{t \in T_1} 1 = \frac{p-3}{2},
	\]
	\item [b)]
	\[
		\sum_{t \in T_1} P(t) =
		\begin{cases}
		\frac{(p-1)^2}{2}, \textrm{ if } p \equiv 1 \pmod{4},\\
		\frac{p^2-2p-3}{2}, \textrm{ if } p \equiv 3 \pmod{4}.
		\end{cases}
	\]
	\item [c)]
	\[
		\ds\sum_{t \in T_1} P(t)^2 =	
		\begin{cases}
			\frac{p^3+1-d(p)}{2}, \textrm{ if } p \equiv 1 \pmod{4},\\
			\frac{p^3-8p-7-d(p)}{2}, \textrm{ if } p \equiv 3 \pmod{4}\\
		\end{cases}
	\]	
\end{itemize}
\end{proposition}
\begin{proof}
a) By definition, $T_1$ is equal to the image of $\F$-points (which are not cusps) on $X(2,4)_{\F}$ under the natural map $g_2: X(2,4) \rightarrow X_1(4)$ of degree $2$. Since we have
$p+1-c(2,4)=2\sum_{t \in T_1}1$ the claim follows.\\
b) We have $\ds\sum_{t\in T_1}P(t)=\frac{1}{2}\sum_{t\in g_1(X(2,4)(\F)) \atop t \notin \textrm{cusps}}P(t)$. Theorem \ref{thm:trace} implies
\begin{eqnarray*}
Trace(Frob_p|W_{1,\ell}^2)&=&-\sum_{t \textrm{ cusp}}Trace(Frob_p|(i_*\mathcal{F})_t)-\sum_{t\ne  \textrm{ cusp}}\left(p+1-P
(t)\right),\\
&=&-\sum_{t \textrm{ cusp}}Trace(Frob_p|(i_*\mathcal{F})_t) -(p+1)(p+1-c(2,4))+\sum_{t \ne \textrm{ cusp}}P(t).\\
\end{eqnarray*}
Since $\dim S_3(\Gamma_1(4)\cap \Gamma^0(2))=0$, it follows $Trace(Frob_p|W_{1,\ell}^2)=0$, and the claim follows.
 Note that we used
\[
		\sum_{t \in \textrm{cusps}} Trace(Frob_p|(i_*\mathcal{F})_t) =
			\begin{cases}
			4, \textrm{ if } p \equiv 1 \pmod{4},\\
			0, \textrm{ if } p \equiv 3 \pmod{4}.\\
			\end{cases}
\]\\
c) We have $\ds\sum_{t\in T_1}P(t)^2=\frac{1}{2}\sum_{t\in g_1(X(2,4)(\F)) \atop t \notin \textrm{cusps}}P(t)^2$. Theorem \ref{thm:trace} implies
\begin{eqnarray*}
Trace(Frob_p|W_{2,\ell}^2)&=&-\sum_{t \in \textrm{ cusps}}Trace(Frob_p|(i_*\textrm{Sym}^2\mathcal{F})_t)-\sum_{t\notin \textrm{cusps}}\left(\left(p+1-P
(t)\right)^2-p\right),\\
&=& -\sum_{t \notin \textrm{cusps}} P(t)^2+2(p+1)\sum_{t\notin \textrm{cusps}}P(t)-(p+1-c(2,4))(p^2+p+1)\\
&-&\sum_{t \in \textrm{cusps}}Trace(Frob_p|(i_*\textrm{Sym}^2\mathcal{F})_t).
\end{eqnarray*}

It follows from Proposition \ref{prop:forms}d) that $Trace(Frob_p|W_{2,\ell}^2)=d(p)$.   The claim follows.
  Note that we used
\[
		\sum_{t \in \textrm{cusps}} Trace(Frob_p|(i_*\textrm{Sym}^2\mathcal{F})_t) = 4.	
\]
\end{proof}

\subsection{$X(4,8)$}
For $t \in \F, t \ne 0,1$, denote by $E_t'$ the elliptic curve $E_t/\langle 2R \rangle$. The curve $E_t'$ is given by the equation $E_t': y^2 = (x-2t)(x+2t)(x-2t+4)$, and is isomorphic to the Legendre elliptic curve $\mathcal{E}_{1-t}$ where $\mathcal{E}_t: y^2 = x(x-1)(x-t)$.

\begin{proposition}\label{prop:no4}
\begin{itemize}
	\item [a)] If  $p \equiv 1 \pmod{4}$ and $t \in T_5$, then $\mathcal{E}_{1-t}$ has full $\F$-rational $4$-torsion, and the point $(1,0) \in \mathcal{E}_{1-t}$ is divisible by $4$.
	
	\item[b)] If $p$ is an odd prime and $t \in T_4$, then the point $(1,0) \in \mathcal{E}_{1-t}$ is divisible by $4$.
\item[c)] If $p\equiv 3 \pmod{4}$ then there are no elliptic curves over $\F$ with full $4$-torsion over $\F$.
\end{itemize}
\end{proposition}
\begin{proof}
a) and b) Let $S\in E_t$ be such that $2S \in E_t(\F)$, $4S=R$ and $\langle S \rangle$ is $\F$-rational. Then $S + \langle 2R \rangle \in E_t/\langle 2R \rangle$ is $\F$-rational and has order $8$. The point $4S + \langle 2R \rangle$ maps to the point $(1,0)\in \mathcal{E}_{1-t}$ under $E_t/\langle 2R \rangle \cong  \mathcal{E}_{1-t}$.

If $p\equiv 1 \pmod{4}$, and $T\in E_t(\F)$ of order $2$, $T \ne 2R$, we have by Proposition \ref{prop:square}a) and c), that $x(T)$ is a square in $\F$ and that $P^\sigma-P \in \{\mathcal{O}, 2R\}$, for all $\sigma \in \textrm{Gal}(\overline{\F}/\F)$, where $2P=T$. It follows that $P + \langle 2R \rangle$ is $\F$-rational of order $4$.\\
\indent c) If $y^2=(x-a)(x-b)(x-c)$ is an elliptic curve over $\F$ with $a,b,c\in \F$, then it follows from the descent homomorphism that, for example, the point $(a,0)$ is divisible by $2$ over $\F$ if and only if $a-b$ and $a-c$ are squares in $\F$.  If both $(a,0)$ and $(b,0)$ are divisible by $2$ (i.e. if the elliptic curve has full $\F$-rational $4$-torsion), then both $a-b$ and $b-a$ are squares, hence $-1$ is a square in $\F$, and $p \equiv 1 \pmod{4}$.
\end{proof}

Modular curve $X(4,8)$ is a moduli space for (generalized) elliptic curves with (linearly independent) points of order $8$ and $4$ with the fixed value of Weil pairing.
We have a map $g:X(4,8) \rightarrow X(2)$, given by the $(E,Q,P) \mapsto (E, 2Q, 4P)$, where $Q$ and $P$ are points on $E$ of order $4$ and $8$ respectively. The degree of this map is $16$ (note that we identify $(E,Q,P)$ with $(E,-Q,-P)$ and take into account only those pairs $(Q,P)$ which satisfy the Weil pairing condition). Denote by $\tilde{g}:X(2,8) \mapsto X(2)$ the map given by $(E,Q,P)\mapsto (E,Q,4P)$.

\begin{proposition}\label{prop:zadnja}
\begin{itemize}
	\item [a)]
	\[
	\sum_{t \in T_5} 1 =
	\begin{cases}
	\frac{p-a(p)-15}{16}, \textrm{ if } p \equiv 1 \pmod{8},\\
	\frac{p-3}{8}, \textrm{ if } p \equiv 3 \pmod{8},\\
	\frac{p-a(p)-7}{16}, \textrm{ if } p \equiv 5 \pmod{8},\\
	\frac{p-7}{8}, \textrm{ if } p \equiv 7 \pmod{8}.\\
	\end{cases}
	\]
	\item [b)]
	\[
		\sum_{t \in T_4} 1 =
			\begin{cases}
			\frac{3p+a(p)-21}{16}, \textrm{ if } p \equiv 1 \pmod{8},\\
			\frac{p-3}{4}, \textrm{ if } p \equiv 3 \pmod{8},\\
			\frac{3p+a(p)-13}{16}, \textrm{ if } p \equiv 5 \pmod{8},\\
			\frac{p-7}{4}, \textrm{ if } p \equiv 7 \pmod{8}.\\
			\end{cases}
	\]
\end{itemize}
\end{proposition}
\begin{proof}
a) Let $p \equiv 1 \pmod{4}$. It follows from Proposition \ref{prop:square} that $\#T_5$ is equal to the number of $t's$ for which the points $(1,0), (0,0) \in \mathcal{E}_t$ are divisible by $4$ and $2$ (in $\mathcal{E}_t(\F)$) respectively, which in turn, by Proposition \ref{prop:square} is equal to the number of the point in the image of $\F$-point of $X(4,8)$ under the map $g$, i.e. $\#T_4=\#g(X(4,8)(\F))$.

Note that $f_1(\tau)$ is the modular form that corresponds under the modularity theorem to the elliptic curve $X(4,8)$, hence $\#X(4,8)(\F) = p+1-a(p)$. Also, if one point in the preimage of $g$ is $\F$-rational, then the same holds for all the points in the preimage, so we have that $p+1-a(p)-c(4,8)=16\sum_{t \in T_5}1$, and the claim follows.\\
If $p\equiv 3 \pmod{4}$ then $\#T_5=\#T_2$ (by Proposition \ref{prop:square} a)), and the claim follows from Proposition $\ref{prop:T2}$. \\
b) Similarly as in part a), it follows from Proposition \ref{prop:square} that $\#T_4$ is equal to the number of elements in the image of $\F$-rational points of $X(2,8)$ under the map $\tilde{g}$.

Let $p \equiv 1 \pmod{4}$.  There are $\#T_5$ points in the image of $\tilde{g}$ which are also in the image of $g$, hence each of these points have eight $\F$-rational points in the preimage by the map $\tilde{g}$, while the remaining $\#T_4-\#T_5$ points have four $\F$-rational points in the preimage. Hence $8\#T_5 + 4\#(T_4-\#T_5)=p+1-c(2,8)$, and the claim follows.

If $p\equiv 3 \pmod{4}$ then by Proposition \ref{prop:no4} there are no elliptic curves over $\F$ with full $4$-torsion over $\F$, hence $4 \#T_4=p+1-c(2,8)$, and the claim follows.
\end{proof}
Next we prove that if $p\equiv 3 \pmod{4}$ then $\#T_0=\frac{1}{2}\#T_1$ and $\sum_{t \in T_0}P(t)=\frac{1}{2}\sum_{t \in T_1}P(t)$. By definition, $t \in T_1$ implies that $1-t=u^2$ for some $u \in \F$. Denote by $t'=1-\left(\frac{1}{u}\right)^2$. It follows that $t'\in T_1$ and $(t')'=t$. Moreover, only one of $t=1-u^2$ and $t'=\frac{u^2-1}{u^2}$ is a square, hence precisely one of them is an element of $T_0$. It follows that $\#T_0=\frac{1}{2}\#T_1$. The second equality now follows from the fact that $P(t)=P(t')$ (it is easy to check that $E_t$ and $E_{t'}$ have the same $j$-invariants).

Theorem \ref{thm:main} now follows from Proposition \ref{prop:main}, Proposition \ref{prop:W1}, Propositions \ref{prop:prva}-\ref{prop:zadnja} and the previous discussion.

\section{Diophantine $m$-tuples in $\F$ and character sums}\label{sec:gen}

In this section, we will use properties of character sums (sums of the Legendre symbols)
to show that for arbitrary $m\geq 2$ there exist Diophantine $m$-tuples in $\F$ for sufficiently
large $p$. We will also derive formulas for the number of Diophantine pairs and triples in $\F$.

\begin{theorem} \label{tm:mtuples}
Let $m\geq 2$ be an integer. If $p > 2^{2m-2} m^2$ is a prime, then there exists a Diophantine $m$-tuple
in $\F$.
\end{theorem}

\begin{proof}
We prove the theorem by induction on $m$. For $m = 2$ and $p > 16$ (in fact, for $p\geq 5$),
we may take the Diophantine pair $\{1,3\}$ in $\F$.

Let $m\ge 2$ be an integer such that the statement holds. Take a prime $p > 2^{2m} (m+1)^2$.
Since $p > 2^{2m-2} m^2$, there exist a Diophantine $m$-tuple
$\{a_1,\ldots,a_m\}$ in $\F$. Let
$$ g:= \# \{x\in \F \,:\, \left(\frac{a_ix+1}{p}\right)=1,\, \mbox{for $i=1,\ldots,m$} \} $$
and denote by $\bar{a_i}$ the multiplicative inverse of $a_i$ in $\F$. Then, by \cite[Exercise 5.64]{LN},
we have
\begin{eqnarray*}
g &\!\!=\!\!& \# \{x\in \F \,:\, \left(\frac{x+\bar{a_i}}{p}\right)=\left(\frac{\bar{a_i}}{p}\right),\, \mbox{for $i=1,\ldots,m$} \} \\
&\!\! \geq \!\!& \frac{p}{2^m} - \left( \frac{m-2}{2} + \frac{1}{2^m}\right) \sqrt{p} - \frac{m}{2}.
\end{eqnarray*}
Since,
\begin{eqnarray*}
\left( \frac{m-2}{2} + \frac{1}{2^m}\right) \sqrt{p} + \frac{m}{2} + (m+1) \leq
\sqrt{p} \left( \frac{m}{2} - 1 + \frac{1}{2^m} + \frac{3}{2^{m+1}} \right)
< \frac{m}{2} \sqrt{p} < \frac{p}{2^m},
\end{eqnarray*}
we get that $g > m+1$. Thus, we conclude that there exist $x\in \F$, $x\not\in \{0,a_1,a_2,\ldots,a_m\}$,
such that $\left(\frac{a_ix+1}{p}\right)=1$ for $i=1,\ldots,m$. Hence,
$\{a_1,\ldots,a_m,x\}$ is a Diophantine $(m+1)$-tuple in $\F$.
\end{proof}

In the proof of the next two propositions we will several times use the following well-known fact
(see e.g. \cite[Section 7.8]{Hua}):
$$ \sum_{x\in \F} \left(\frac{\alpha x^2 + \beta x + \gamma}{p} \right) = -\left(\frac{\alpha}{p}\right), $$
provided $\beta^2 - 4\alpha \gamma \not\equiv 0 \pmod{p}$.

\begin{proposition} \label{prop:pairs}
Let $p$ be an odd prime. The number of Diophantine pairs in $\F$ is equal to
\[
N^{(2)}(p) =
			\begin{cases}
			\frac{(p-1)(p-2)}{4}, \textrm{ if } p \equiv 1 \pmod{4},\\
			\frac{p^2-3p+4}{4}, \textrm{ if } p \equiv 3 \pmod{4}.\\
			\end{cases}
\]
\end{proposition}

\begin{proof}
We have
$$ 4N^{(2)}(p) = \sum_{a,b\neq 0, a\neq b} \left( 1+ \left(\frac{ab+1}{p}\right)'\right), $$
where $\left(\frac{x}{p}\right)'=\left(\frac{x}{p}\right)$ for $x\neq 0$ and
$\left(\frac{0}{p}\right)'=1$.
Therefore, we have
\begin{eqnarray*}
4N^{(2)}(p) &=& \sum_{b\neq 0} \sum_{a\neq 0,b} 1 +
\sum_{b\neq 0} \sum_{a\neq 0,b} \left(\frac{ab+1}{p}\right) +
\sum_{b\neq 0, b^2\neq -1} 1 \\
&=& (p-1)(p-2) + \sum_{b\neq 0} \left (-1 - \left(\frac{b^2+1}{p}\right) \right) + \sum_{b\neq 0, b^2\neq -1} 1.
\end{eqnarray*}
If $p\equiv 1\pmod{4}$, the last sum is equal to $p-3$. Thus we get
$$  4N^{(2)}(p) = (p-1)(p-2) -(p-1)+2+(p-3)=(p-1)(p-2). $$
Similarly, for $p\equiv 3\pmod{4}$, we get
$$  4N^{(2)}(p) = (p-1)(p-2) -(p-1)+2+(p-1)=p^2-3p+4. $$
\end{proof}

\begin{proposition} \label{prop:triples}
Let $p$ be an odd prime. The number of Diophantine triples in $\F$ is equal to
\[
N^{(3)}(p) =
			\begin{cases}
			\frac{(p-1)(p-3)(p-5)}{48}, \textrm{ if } p \equiv 1 \pmod{4},\\
			\frac{(p-3)(p^2-6p+17)}{48}, \textrm{ if } p \equiv 3 \pmod{4}.\\
			\end{cases}
\]
\end{proposition}

\begin{proof}
We have
$$ 48N^{(3)}(p) = \sum_{S} \left( 1+ \left(\frac{ab+1}{p}\right)'\right)\left( 1+ \left(\frac{ac+1}{p}\right)'\right)\left( 1+ \left(\frac{bc+1}{p}\right)'\right), $$
where the sum is taken over all triples $a,b,c$ in $\F$ such that
$a,b,c\neq 0$, $a\neq b$, $a\neq c$, $b\neq c$.
Let us denote:
\begin{eqnarray*}
S_1 &=& \sum_{S} 1, \\
S_2 &=& \sum_{S} \left(\frac{ab+1}{p}\right), \\
S_3 &=& \sum_{S} \left(\frac{ab+1}{p}\right)\left(\frac{ac+1}{p}\right), \\
S_4 &=& \sum_{S} \left(\frac{ab+1}{p}\right)\left(\frac{ac+1}{p}\right)\left(\frac{bc+1}{p}\right), \\
S_5 &=& \sum_{S'} 1, \\
S_6 &=& \sum_{S'} \left(\frac{ab+1}{p}\right), \\
S_7 &=& \sum_{S'} \left(\frac{ab+1}{p}\right)\left(\frac{-a^{-1}b+1}{p}\right),
\end{eqnarray*}
where the sums $S_5,S_6,S_7$ are taken over all pairs $a,b$ in $\F$ such that
$a,b\neq 0$, $b\neq a,-a^{-1}$, $a^2\neq -1$.
Then we have
\begin{equation} \label{eq:N3}
48N^{(3)}(p) = S_1 + 3S_2 + 3S_3 + S_4 + 3S_5 + 6S_6 + 3S_7.
\end{equation}
Thus, it remains to compute the sums $S_1,\ldots,S_7$. We will derive the formulas for the
cases $p\equiv 1,3\pmod{4}$. When the formulas for these two cases differ,
the upper sign will correspond to $p\equiv 1\pmod{4}$, while the lower sign will correspond
to $p\equiv 3\pmod{4}$.

We have
$$ S_1= (p-1)(p-2)(p-3), $$
\begin{eqnarray*}
S_2 = (p-3)\sum_{\substack{a,b\neq 0\\a\neq b}} \left(\frac{ab+1}{p}\right) = -(p-3)^2,
\end{eqnarray*}
\begin{eqnarray*}
S_3 &=& \sum_{\substack{a,b\neq 0\\a\neq b}} \left(\frac{ab+1}{p}\right) \sum_{c\neq 0,a,b} \left(\frac{ac+1}{p}\right) \\
&=& \sum_{\substack{a,b\neq 0\\a\neq b}} \left (-1-\left(\frac{ab+1}{p}\right) -\left(\frac{a^2+1}{p}\right) \right) \\
&=& (p-3) - \sum_{\substack{a,b\neq 0\\a\neq b\\ ab+1\neq 0}} 1 - \sum_{\substack{a,b\neq 0\\a\neq b}}
\left(\frac{ab+1}{p}\right) \left(\frac{a^2+1}{p}\right) \\
&=& (p-3) - (p^2-4p+4 \pm 1) - \sum_{a\neq 0} \left(\frac{a^2+1}{p}\right)
\left(-1-\left(\frac{a^2+1}{p}\right)\right) \\
&=&
(p-3) - (p^2-4p+4 \pm 1) - 2 + (p-2\mp 1) = -p^2 + 6p -11 \mp 2,
\end{eqnarray*}

\begin{eqnarray*}
S_4 &=& \sum_{\substack{a,b\neq 0\\a\neq b}} \left(\frac{ab+1}{p}\right)\sum_{c\neq 0,a,b}
\left(\frac{ac+1}{p}\right) \left(\frac{bc+1}{p}\right) \\
&=& \sum_{\substack{a,b\neq 0\\a\neq b}} \left(\frac{ab+1}{p}\right) \left(-\left(\frac{ab}{p}\right)-1-
\left(\frac{a^2+1}{p}\right)\left(\frac{ab+1}{p}\right)-
\left(\frac{b^2+1}{p}\right)\left(\frac{ab+1}{p}\right)\right) \\
&=& -\sum_{a\neq 0}\sum_{\substack{t\neq 0, a^2}} \left(\frac{t+1}{p}\right)\left(\frac{t}{p}\right) +
\sum_{a\neq 0} \left(1 + \left(\frac{a^2+1}{p}\right)\right) -
2\sum_{\substack{a,b\neq 0\\a\neq b\\ ab+1\neq 0}} \left(\frac{a^2+1}{p}\right) \\
&=& -\sum_{a\neq 0} \left(-1 - \left(\frac{a^2+1}{p}\right)\right) +
\sum_{a\neq 0} \left(1+\left(\frac{a^2+1}{p}\right)\right) \\
&& \mbox{}-
2\sum_{a\neq 0} \left(-1-1-\left(\frac{a^2+1}{p}\right)-\left(\frac{a^{-2}+1}{p}\right)\right) \\
&=& (4p-12)+2((p-1)-2)= 6p-18,
\end{eqnarray*}
$$ S_5 = p^2-5p+6 \mp (p-3), $$
\begin{eqnarray*}
S_6 = \sum_{\substack{a\neq 0\\a^2\neq -1}} \left(-1-\left(\frac{a^2+1}{p}\right)\right) = -p+4 \pm 1,
\end{eqnarray*}
\begin{eqnarray*}
S_7 &=& \sum_{\substack{a\neq 0\\a^2\neq -1}} \sum_{c\neq 0,a,-a^{-1}}
\left(\frac{ac+1}{p}\right) \left(\frac{-a^{-1}c+1}{p}\right) \\
&=& \sum_{\substack{a\neq 0\\a^2\neq -1}} (\mp 1 -1) = -p +3 \mp (p-3).
\end{eqnarray*}
Putting all these formulas together in (\ref{eq:N3}), we get
$$ 48N^{(3)}(p)= p^3 - 9p^2 + 29p - 33 \mp (6p-18), $$
and by writing separately the cases $p\equiv 1,3 \pmod{4}$,
we obtain the formula for $N^{(3)}(p)$ given in the statement of the proposition.
\end{proof}

For small values for $m$, the bound from Theorem \ref{tm:mtuples} can be improved
by using concrete examples of integer Diophantine $m$-tuples for $m=2,3,4$ and
rational Diophantine $m$-tuples for $m=5,6$.
From the integer Diophantine pair $\{2,4\}$ we get $N^{(2)}(p)> 0$ for $p\geq 3$;
from the integer Diophantine triple $\{2,4,12\}$ we get $N^{(3)}(p) >0$ for $p\geq 7$;
from the integer Diophantine quadruples $\{1,3,8,120\}$ and $\{2,24,40,7812\}$
we get $N^{(4)}(p) >0$ for $p\geq 11$ (for any prime $p\geq 11$ at least one
of these two quadruples gives, by the reduction modulo $p$, a Diophantine quadruple
in $\F$);
from the rational Diophantine quintuples
$\{5/16, 21/16, 4, 285/16, 420\}$ and $\{1/5, 21/20, 69/20, 25/4, 96/5\}$
(see \cite{D-quint}) we get
$N^{(5)}(p) >0$ for $p\geq 23$;
from the rational Diophantine sextuples
$\{221/1260, 175/324, 203/180, 81/35, 265/28, 1120/9\}$,
$\{377/1260, 119/180,$ $297/140, 992/315, 175/9, 2275/4\}$,
$\{5/36, 665/1521, 5/4, 32/9, 3213/676, 189/4\}$ (see \cite{Gibbs})
we get $N^{(6)}(p) >0$ for $p\geq 43$.

We can follow the proof of Proposition \ref{prop:triples}
to sketch the proof of the asymptotic formula $N^{(m)}(p)=\ds \frac{1}{2^{m \choose 2 }}\frac{p^m}{m!} + o(p^m)$.
Indeed, we have
$$  m! {2^{m \choose 2 }} N^{(m)}(p) = \sum \prod_{1\leq i <j \leq m}
\left(1+ \left(\frac{a_i a_j+1}{p} \right)' \right), $$
where the sum is taken over all $m$-tuples $a_1,\ldots ,a_m$ of distinct non-zero elements of $\F$.
The main term comes from $\sum 1 = (p-1)(p-2)...(p-m)=p^m + o(p^m)$,
while all other terms are of the form
$$ \sum_{a_1,...,a_{m-1}} \sum_{a_m} \left(\frac{f(a_m)}{p}\right), $$
where $f(x)$ is a non-square polynomial of degree $\leq m-1$ and the sums are taken
over almost all $m$-tuples in $\F$. By Weil's estimate for character sums (see e.g. \cite[Theorem 5.41]{LN}),
we conclude that the contribution of all these terms is $O(p^{m-1}\sqrt{p}) = o(p^m).$

\bigskip

\begin{acknowledgements}
{We would like to thank Ivica Gusi\'c, Filip Najman and Anthony Scholl for some helpful comments.
This work was supported by the QuantiXLie Centre of Excellence, a project
cofinanced by the Croatian Government and European Union through the
European Regional Development Fund - the Competitiveness and Cohesion
Operational Programme (Grant KK.01.1.1.01.0004).
A.D. was supported by the Croatian Science Foundation under the project no. 6422.}
\end{acknowledgements}

\end{document}